\documentclass[10pt]{article}

\usepackage[margin=1in]{geometry}
\usepackage{amsmath,amssymb,amsthm}
\usepackage{microtype}
\usepackage{setspace}
\usepackage{xcolor}
\usepackage{titlesec}
\titleformat{\section}{\normalfont\large\bfseries}{\thesection}{1em}{}
\titlespacing*{\section}{0pt}{2ex plus 1ex minus .2ex}{1ex plus .2ex}
\newtheorem{theorem}{Theorem}
\newtheorem{lemma}{Lemma}
\newcommand{\bE}{\mathbb{E}}

\newcommand{\cH}{\mathcal{H}}

\newcommand{\Pb}{\mathbb P}
\newcommand{\Var}{\operatorname{Var}}
\newcommand{\Cov}{\operatorname{Cov}}

\title{\vspace{-0.7cm} Perfect matchings in hypergraphs and Feige's inequality}
\author{Aleksa Milojevi\'c\thanks{Department of Mathematics, ETH, Z\"urich, Switzerland. Research supported in part by SNSF grant 200021-228014. {Email}: {\tt \{aleksa.milojevic, benjamin.sudakov\}@math.ethz.ch}} \and Benny Sudakov\footnotemark[1]}
\date{}

\begin{document}
\setstretch{1.2}

\maketitle
\vspace{-0.2cm}

How large of a minimum degree does an $n$-vertex graph need before we are sure that it contains a perfect matching? Dirac's theorem states that a graph on an even number of vertices in which each vertex has degree at least $n/2$ has this property. In this short expository note, intended to be used in the classroom, we discuss how this statement generalizes to hypergraphs. In particular, we highlight an elegant connection between fractional perfect matchings in hypergraphs and a probabilistic inequality about nonnegative random variables, which was conjectured by Feige. We also present a very short self-contained proof of Feige's conjecture. 

A $k$-uniform hypergraph $\cH$ is a collection of sets of size $k$, called \textit{hyperedges}, on a ground set of size $n$, denoted by $V(\cH)$. For an integer $1\leq d\leq k$, the degree of a $d$-tuple of vertices $T\subseteq V(\cH)$, denoted by $\deg_\cH T$ is the number of hyperedges containing $T$, and the minimum $d$-degree of $\cH$ is the minimum degree of some $d$-tuple of vertices $\delta_d(\cH)=\min_{T\subseteq V(\cH), |T|=d} \deg_\cH T$. A \textit{matching} is a collection of disjoint edges in $\cH$, and a matching is perfect if it covers all vertices of $\cH$. For a pair of integers $1\leq d\leq k$, we ask:
\vspace{-0.2cm}
\[\text{What minimum $d$-degree of a hypergraph $\cH$ guarantees that $\cH$ always has a perfect matching?} \] 
\vspace{-0.8cm}

\noindent
The first step in answering this question is to determine what minimum degree guarantees that $\cH$ contains a \textit{fractional} perfect matching. A \textit{fractional matching} is an assignment of weights $\nu(e)\in [0, 1]$ to all hyperedges of $\cH$ such that for every vertex $v\in V(\cH)$ we have $\sum_{e:v\in e}\nu(e)\leq 1$. Such a matching is \textit{perfect} if the weight of all hyperedges satisfies $\sum_{e\in E(\cH)}\nu(e)=n/k$ (which is the maximum possible value). Fractional matchings are a relaxation of the usual notion, which corresponds to the case when $\nu(e)\in \{0, 1\}$ for every hyperedge $e$.

\begin{theorem}\label{thm:fractional_matching}
Let $1\leq d\leq k$ be integers and let $\cH$ be the hypergraph with $\delta_d (\cH)> \Big(1-\big(1-\frac{1}{k}\big)^{k-d}\Big)\frac{n^{k-d}}{(k-d)!}$. Then $\cH$ has a perfect fractional matching.
\end{theorem}

Since the maximal number of edges containing a $d$-tuple in a hypergraph $\cH$ is $\binom{n-d}{k-d}$, Theorem~\ref{thm:fractional_matching} is often phrased in the following, almost equivalent, way. For every $\varepsilon>0$, integers $1\leq d\leq k$ and sufficiently large $n$, if a hypergraph $\cH$ has $\delta_d(\cH)\geq \big(1-(1-1/k)^{k-d}+\varepsilon)\binom{n-d}{k-d}$, then it contains a fractional perfect matching.

The proof of this theorem is based on Feige's inequality, which we now motivate. If $X$ is a nonnegative random variable of expectation $1$, then the probability that $X$ falls below some threshold $\lambda$ can be estimated by Markov's inequality, which states that $\Pb[X< \lambda]\geq 1-\frac{1}{\lambda}$. This is tight in the case where $X$ takes only two values, $0$ and $\lambda$. However, if we replace $X$ by a sum of nonnegative random variables, it turns out that the conclusion can be strengthened significantly. Namely, if $X_1, \dots, X_n$ are independent nonnegative random variables of expectation $1$, then Feige's inequality gives a lower bound on $\Pb[X_1+\dots+X_n<\lambda]$. Since the expectation of the sum $X_1+\dots+X_n$ is $n$, the common way to parametrize the threshold is to write $\lambda=n+\delta$. 

\begin{theorem}\label{thm:feige}
Let $X_1,\ldots,X_n$ be independent random variables with $\bE X_i\leq 1$ and $X_i\geq 0$ and let $\delta\geq 1$. Then
\vspace{-0.2cm}
\[
 \Pb\Big(\sum_{i=1}^n X_i<n+\delta\Big) \geq \Big(1-\frac{1}{n+\delta}\Big)^n
\]
\end{theorem}

In Section~\ref{sec:matchings} we discuss the problem of hypergraph matchings in more detail and give the proof of Theorem~\ref{thm:fractional_matching}. In Section~\ref{sec:feige} we discuss the proof of Feige's inequality, and give some additional remarks.

\section{Perfect fractional matchings}\label{sec:matchings}

The connection between perfect matchings and Feige's inequality presented here appears in the work of Alon, Frankl, Huang, R\"odl, Ruci\'nski and Sudakov \cite{AFHRRS12}. Further, they showed how one can extend the proof to to usual perfect matchings, as long as $\delta_d(\cH)\geq \big(\max\{\frac{1}{2}, 1-(1-\frac{1}{k})^{k-d}\}+\varepsilon\big)\binom{n-d}{k-d}$, using the absorption method. The constant $1/2$ in the maximum above is the so-called \textit{divisibility barrier} and it necessary for the absorption method. Although subsequent works \cite{CLZ26,HLWY26,TZ13} sharpened these asymptotic results and found the precise threshold, the value of the constant $1-(1-1/k)^{k-d}$ is ultimately determined by this argument.

\begin{proof}[Proof of Theorem~\ref{thm:fractional_matching}.]
If the largest fractional matching in $\cH$ has weight less than $n/k$, by linear programming duality there is a fractional vertex cover of $\cH$ with total weight less than $n/k$. In other words, there are weights $w_v$ for all vertices $v\in V(\cH)$ such that for each hyperedge $e$ we have $\sum_{v\in e} w_v\geq 1$ and $\sum_{v\in V(\cH)}w_v<n/k$.

Let $T$ be the set of $d$ vertices with the smallest weight, and note that their average weight $t=\frac{1}{d}\sum_{v\in T}w_v$ is smaller than $1/k$ (since the average weight of all vertices is $<1/k$). Moreover, any vertex $v\notin T$ has $w_v\geq t$. 

Let us now define a nonnegative random variable $X$ as follows: pick a uniformly random vertex $v\in V(\cH)$ and put $X=0$ if $v\in T$ and $X=\frac{k}{1-kt}(w_v-t)$ if $v\notin T$. To bound the expectation of $X$, observe that
\vspace{-0.2cm}
\begin{equation}\label{eq:expectation_X}
    \bE\Big[\frac{1-kt}{k}X\Big]=\frac{1}{n}\sum_{v\in V(\cH)}(w_v-t)<\frac{1}{n}(n/k-nt)=1/k-t= \frac{1-kt}{k}.\vspace{-0.2cm}
\end{equation}

Note that the contributions to $\bE[\frac{1-kt}{k}X]$ of vertices of $T$ are $0$ and not $w_v-t$ as written above, but this makes no difference since $\sum_{v\in T}(w_v-t)=\sum_{v\in T}0$. The equation~(\ref{eq:expectation_X}) shows $\bE X\leq 1$.

Take $k-d$ independent copies of $X$, denoted by $X_1, \dots, X_{k-d}$, and say the corresponding vertices of $\cH$ are $v_1, \dots, v_{k-d}$. By Theorem~\ref{thm:feige} applied to $k-d$ random variables and $\delta=d\geq 1$, we have 
\vspace{-0.2cm}
\begin{equation}\label{eq:feige_applied}
\Pb\Big[\sum_{j=1}^{k-d} X_j< (k-d)+d\Big]\geq \Big(1-\frac{1}{(k-d)+d}\Big)^{k-d}=\Big(1-\frac{1}{k}\Big)^{k-d}.\vspace{-0.2cm}
\end{equation}
\vspace{-0.2cm}

The punchline is that whenever the vertices $v_1, \dots, v_{k-d}$ form a hyperedge together with $T$, we have $dt+\sum_{j=1}^{k-d}w_{v_j}\geq 1$ (since $w$ is a fractional cover), and so we have $\sum_{j=1}^{k-d} X_j=\frac{k}{1-kt}\sum_{j=1}^{k-d}(w_{v_j}-t)\geq \frac{k}{1-kt}\big(1-dt-(k-d)t\big)\geq k$. So, whenever $\sum_{j=1}^{k-d} X_j<k$, the set $\{v_{1}, \dots, v_{k-d}\}\cup T$ is not a hyperedge. Hence, (\ref{eq:feige_applied}) implies that there are at least $\big(1-\frac{1}{k}\big)^{k-d}n^{k-d}$ ways to choose an ordered $(k-d)$-tuple of vertices $(v_1, \dots, v_{k-d})$ which does not form a hyperedge with $T$. 

On the other hand, by the assumption on $\delta_d(\cH)$, there are at least $(k-d)!\delta_d(\cH)> \big(1-\big(1-\frac{1}{k}\big)^{k-d}\big)n^{k-d}$ ways to choose an ordered $(k-d)$-tuple of vertices $(v_1, \dots, v_{k-d})$ which forms an edge together with $T$. So, in total, there are more than $n^{k-d}$ ways to choose a $(k-d)$-tuple of vertices $(v_1, \dots, v_{k-d})$, a contradiction. 
\end{proof}

Note that Theorem~\ref{thm:feige} which we use requires the gap $\delta$ between the expectation of $\sum_{i=1}^n X_i$ and the threshold $n+\delta$ to be at least $\delta\geq 1$. So, at first glance, it may be surprising that the above proof is able to get any mileage out of the fact that the best fractional cover has weight $<n/k$, without any lower bound on how large the gap between $\sum_{v}w_v$ and $n/k$ is. In some sense, the key point is that in the definition of the random variables $X_i$, we rescale the weights $w_v$, to ensure that the relevant gap is $d$ (see equation~(\ref{eq:feige_applied})).

The constant $1-(1-1/k)^{k-d}$ in Theorem~\ref{thm:fractional_matching} is best possible, as the so-called \textit{space barrier} shows. Take $A\subseteq V(\cH)$ to be a set of $n/k-1$ vertices (where $k$ divides $n$ say), and consider the hypergraph $\cH$ consisting of those $k$-sets which intersect $A$. The heaviest fractional matching $\nu$ of $\cH$ has weight less than $n/k$. To see this, note that $\sum_{e\in E(\cH)}\nu(e)\leq \sum_{v\in A} \sum_{e\ni v} \nu(e)\leq \sum_{v\in A} 1=|A|<n/k$. The first inequality follows since every edge of $\cH$ contains some vertex of $A$, while the second one just applies the definition of the fractional matching to say that $ \sum_{e\ni v} \nu(e)\leq 1$ for any $v\in V(\cH)$.\footnote{This argument is a fractional version of the perhaps more familiar argument which says that if every hyperedge intersects the set $A$, then there do not exist $|A|+1$ disjoint edges.} The minimum $d$-degree of this hypergraph is determined by the number of edges containing a $d$-tuple of vertices disjoint from $A$, which is \vspace{-0.2cm}\[\delta_d(\cH)\!=\!\binom{n\!-\!d}{k\!-\!d}-\binom{n\!-\!d\!-\!\frac{n}{k}\!+\!1}{k-d}\!=\!(1+o(1)) \Bigg[\frac{n^{k-d}}{(k\!-\!d)!}-\big(1-\frac{1}{k}\big)^{k-d}\frac{n^{k-d}}{(k\!-\!d)!}\Bigg]\!=\!\Big(1+o(1)-\big(1-\frac{1}{k}\big)^{k-d}\Big)\frac{n^{k-d}}{(k\!-\!d)!}.\] 

\section{Proof of Feige's inequality}\label{sec:feige}

Feige's inequality was conjectured in a slightly different form from Theorem~\ref{thm:feige} by Feige in \cite{F06}. He asked to show that for all $\delta>0$, we have $\Pb\big[\sum_{i=1}^n X_i<n+\delta\big]\geq \min\{\frac{\delta}{\delta+1}, 1/e\}$ for independent nonnegative variables $X_1, \dots, X_n$ with mean $1$. After several weaker results and special cases of the conjecture were addressed \cite{AAMM24, EF25, G20, GHLL20, HZZ10}, the conjecture was proven in the papers of Fu, Han, Wang, Yan, Zhang and Zhou \cite{FHWYZZ26} and Nie and Wei \cite{NW26}. Both articles relied on external input \cite{VT26}, which motivated us to write the exposition of a short and self-contained proof of Feige's inequality. This proof is inspired by Han's elegant simplification \cite{Han26} of the proof of Samuels' conjecture from Ling's paper \cite{Ling26}. Samuels' conjecture \cite{S66} is a more general version of Feige's conjecture, in which the means of the variables $X_i$ do not all have to be the same.

\begin{proof}[Proof of Theorem~\ref{thm:feige}.]
By increasing the random variables, we may assume that $\bE X_i=1$ for all $i$.
A standard reduction, originally observed by Feige \cite{F06}, allows us to assume that random variables $X_i$ are supported on at most two values, as shown in Lemma~\ref{lemma:two_values}. For completeness, we present the proof of this statement after the main proof. It will be convenient to define $\lambda=n+\delta$ and $S=\sum_{i=1}^n X_i$. Our goal will be to give a lower bound on $\Pb[S<\lambda]$.

We will perform a sequence of \textit{moves}, where in each move we replace some random variable $X_i$ by its lower value, denoted by $u_i$. After such a move, we obtain the sum $S_i=\sum_{j\neq i} X_j + u_i$. We observe that $S_i\leq S$, meaning that $\Pb[S_i<\lambda]\geq \Pb[S<\lambda]$. Our first lemma shows a lower bound on $\Pb[S<\lambda]$ in terms of $\Pb[S_i<\lambda]$. 

To state the exact bound, let us introduce two pieces of notation. For a random variable $S$, we denote $r(S)=\lambda-\bE[S|S<\lambda]$ and $b(S)=\bE[S]-\bE[S|S<\lambda]$. To motivate these two quantities, we recommend the reader to look at the proof of Lemma~\ref{lemma:lower_bound}. As we will see, the larger $r(S)$ is, the better the lower bound provided by Lemma~\ref{lemma:lower_bound} will be. Therefore, throughout the replacement process, we would like the value of $r(S)$ to increase, motivating the following definition. A move is \textit{safe} if $r(S_i)\geq r(S)$. Then, we have the following lower bound on $\Pb[S<\lambda]$.

\begin{lemma}\label{lemma:lower_bound}
For any move $S\mapsto S_i$, we have
\vspace{-0.1cm}
\begin{equation}\label{eq:lower_bound}
\Pb[S<\lambda]\ge \Pb[S_i<\lambda]\Big(1-\frac{b(S)-b(S_i)}{r(S)}\Big).
\end{equation}
If the move is safe, we also have $0\leq b(S)-b(S_i)\leq 1$.
\end{lemma}

Our second lemma shows that there is always a safe move available.

\begin{lemma}\label{lemma:safe_move}
Unless $S$ is deterministic, there is some $i$ such that replacing $X_i$ by $u_i$ is a safe move.
\end{lemma}
 
Let us first show how to complete the proof using the above lemmas, and then we will present their proofs. We start with $S_0=S$ and repeatedly make a safe move until all variables are fixed. This produces a sequence of sums $S_0, S_1, S_2, \dots, S_n$, where $S_n$ is deterministic and at most $n<\lambda$. Since every move is safe, we have $r(S_t)\geq r(S)$ for all $t$ and therefore by repeated application of Lemma~\ref{lemma:lower_bound} we get that
\vspace{-0.1cm}
\[\Pb[S<\lambda]\ge \Pb[S_n<\lambda] \prod_{t=1}^n\left(1-\frac{b(S_{t-1})-b(S_{t})}{r(S_{t-1})}\right)\geq \prod_{t=1}^n\left(1-\frac{b(S_{t-1})-b(S_{t})}{r(S)}\right),\]
since $\Pb[S_n<\lambda]=1$ and $r(S_{t-1})\geq r(S)$. 

Let us denote $b=b(S)\in [0, n]$, and observe that $r(S)=\lambda-\bE[S|S<\lambda]=n+\delta+b(S)-\bE[S]=\delta+b$. Moreover, if we denote the summands $b(S_{t-1})-b(S_t)$ by $\Delta_t$, we observe that they telescope $\sum_{t=1}^n\Delta_t=b(S_0)-b(S_n)=b$ and satisfy $\Delta_t\in [0, 1]$. Hence, we have
\vspace{-0.1cm}
\[\Pb[S<\lambda]\ge \prod_{t=1}^n\left(1-\frac{\Delta_t}{\delta+b}\right)\ge \left(1-\frac{1}{\delta+b}\right)^{\sum_{t=1}^{n}\Delta_t}=\left(1-\frac{1}{\delta+b}\right)^{b},\]
where the second inequality follows from the fact that $1-\Delta_t/(\delta+b)\geq \big(1-1/(\delta+b)\big)^{\Delta_t}$, by Bernoulli's inequality (which states that for reals $r\in [0, 1]$ and $x>-1$ we have $1+xr\geq (1+x)^r$).

To obtain the final bound, it remains to show that the function $(1-1/(\delta+b))^b$ is minimized when $b=n$. This is the only place where we use the assumption $\delta\geq 1$. By Bernoulli's inequality with exponents at least one (which states that for $r\geq 1$ and $x>-1$ we have $(1+x)^r\geq 1+rx$), we have that 
\vspace{-0.1cm}
\[\Big(1+\frac{1}{n+\delta-1}\Big)^{n/b}\geq 1+\frac{n/b}{n+\delta-1}= 1+\frac{1}{b+(\delta-1)b/n}\geq 1+\frac{1}{b+\delta-1}.\]
\vspace{-0.5cm}

\noindent
Taking reciprocals shows that $1-\frac{1}{b+\delta}\geq \big(1-\frac{1}{n+\delta}\big)^{n/b}$, and so $\Pb[S<\lambda]\geq \Big(1-\frac{1}{\delta+b}\Big)^{b}\geq \Big(1-\frac{1}{n+\delta}\Big)^n$. It remains to show the proofs of the two auxiliary lemmas.

\begin{proof}[Proof of Lemma~\ref{lemma:lower_bound}.]
The goal of this proof is to control how the probability $\Pb[S<\lambda]$ increases when $X_i$ is replaced by $u_i$. To achieve this, we want an upper bound on the weight of outcomes of $S_i$ which land below $\lambda$, where the corresponding outcome of $S$ did not land below $\lambda$. There is a simple way to bound this using the conditional averages, as follows. 

Consider the event $E_i=\{S_i<\lambda\}$, having probability $\Pb[S_i<\lambda]$. This event naturally splits into $E=\{S<\lambda\}$, the outcomes where the original sum was below the threshold, and the newly added outcomes $E_i\backslash E$ where we have $S_i<\lambda$ and $S>\lambda$. The key point is that every new outcome raises the expectation of $S$ on $E_i$ significantly. So, an upper bound on $\bE[S|E_i]$ can give us an upper bound on the weight of the outcomes in $E_i\backslash E$. Putting this into equations, we have
\vspace{-0.2cm}
\begin{equation}\label{eq:increase_in_expectation}
\bE[S|E_i]\geq \frac{\Pb[E]}{\Pb[E_i]}\bE[S|E]+\frac{\Pb[E_i\backslash E]}{\Pb[E_i]}\lambda=\bE[S|E]+\frac{\Pb[E_i\backslash E]}{\Pb[E_i]}\big(\lambda-\bE[S|E]\big).
\end{equation}
\vspace{-0.4cm}

\noindent
This argument clarifies why $r(S)=\lambda-\bE[S|E]$ is the correct quantity to consider. Namely, on every outcome in $E_i\backslash E$, the value of $S$ is by at least $\lambda-\bE[S|E]$ larger than the previous expectation of $S$, namely $\bE[S|E]$. 

Inequality~(\ref{eq:lower_bound}) follows by rearranging (\ref{eq:increase_in_expectation}) to obtain a lower bound on $\Pb[S_i<\lambda]=\Pb[E_i]$. From (\ref{eq:increase_in_expectation}), we have 

\begin{equation}\label{eq:rearrangement}
    \frac{\Pb[E_i\backslash E]}{\Pb[E_i]}\big(\lambda-\bE[S|E]\big)\leq \bE[S|E_i]-\bE[S|E]=\bE[S_i|E_i]-\bE[S|E]+\bE[(X_i-u_i)|E_i],
\end{equation}
\vspace{-0.4cm}

\noindent
where we used the linearity of expectation in the last step, as $S=S_i+(X_i-u_i)$. Since $X_i$ and $E_i$ are independent, the right hand side can be conveniently rearranged as $\bE[S_i|E_i]-\bE[S|E]+\bE[X_i-u_i]=\bE[S_i|E_i]-\bE[S|E]+\bE[S]-\bE[S_i]=b(S)-b(S_i)$. Hence, from (\ref{eq:rearrangement}), we get that 
\vspace{-0.2cm}
\begin{equation*}
   \frac{\Pb[E_i\backslash E]}{\Pb[E_i]}= 1-\frac{\Pb[E]}{\Pb[E_i]}\geq \frac{b(S)-b(S_i)}{r(S)}, \quad \text{ and so } \quad \frac{\Pb[E]}{\Pb[E_i]}\geq 1-\frac{b(S)-b(S_i)}{r(S)}.
\end{equation*}
from where the final inequality (\ref{eq:lower_bound}) directly follows.

Since $S_i\le S$, we have $\Pb[S_i<\lambda]\ge \Pb[S<\lambda]$, and \eqref{eq:lower_bound} therefore implies $b(S)-b(S_i)\ge 0$. If the move is also safe, then $b(S)-b(S_i)=r(S)-r(S_i)+\bE[S]-\bE[S_i]\le \bE[S]-\bE[S_i]=\bE[X_i-u_i]=1-u_i\leq 1$ (since $X_i\geq 0$ always implies $u_i\geq 0$).
\end{proof}

\begin{proof}[Proof of Lemma~\ref{lemma:safe_move}.]
In this proof, we will show that replacing $X_i$ by $u_i$ is a safe move precisely when $X_i$ is positively correlated with the variable $S$ conditioned on $S<\lambda$. This makes sense: removing such a variable $X_i$ certainly decreases the expectation $\bE[S|S<\lambda]$, and consequently increases the difference $r(S)=\lambda-\bE[S|S<\lambda]$. We now give the details.

The key observation is that for every random variable $X_i$ appearing in $S$ we have $\Cov(S, X_i|S<\lambda)=(v_i-u_i)\Pb[X_i=u_i|S<\lambda](r(S_i)-r(S))$, where $u_i<v_i$ is the support of $X_i$.

Once we prove this, we simply write \vspace{-0.1cm}\[0\leq \Var(S|S<\lambda)=\sum_{i=1}^n \Cov(S, X_i|S<\lambda)=\sum_{i=1}^n (v_i-u_i)\Pb[X_i=u_i|S<\lambda](r(S_i)-r(S)),\] where the sums run over the remaining random summands. Since the coefficients $(v_i-u_i)\Pb[X_i=u_i|S<\lambda]$ are positive, this shows that for some $i$ we have $r(S_i)\geq r(S)$, ultimately showing there is a safe move.

Let us now prove the key observation. We can write $X_i=u_i+(v_i-u_i)B$ for a Bernoulli variable $B$, and so $\Cov(S, X_i|S<\lambda)=(v_i-u_i)\Cov(S, B|S<\lambda)$. The latter covariance can be computed as follows
\vspace{-0.1cm}
\begin{align*}
\Cov\!\big(S, B\big|S<\lambda\big)\!=\!-\!\Cov\!\big(S, 1\!-\!B\big| S\!<\!\lambda\big)\!&=\!\bE\big[S\big|S\!<\!\lambda\big]\bE\big[1\!-\!B\big|S\!<\!\lambda\big]\!-\!\bE\big[S(1-B)\big|S<\lambda\big]\\
&=\!\bE\big[S\big|S\!<\!\lambda\big]\Pb\big[B\!=\!0\big|S\!<\!\lambda\big]\!-\!\bE\big[S\big|B\!=\!0\text{ and }S\!<\!\lambda\big]\Pb\big[B\!=\!0\big|S\!<\!\lambda\big]
\end{align*}
When $B=0$ we have $X_i=u_i$ and $S=S_i$. This means that 
$\bE\big[S\big|B\!=\!0\text{ and }S\!<\!\lambda\big]=\bE[S_i|S_i<\lambda]$ and therefore
$\bE\big[S\big|S\!<\!\lambda\big]\!-\!\bE\big[S\big|B\!=\!0\text{ and }S\!<\!\lambda\big]=r(S_i)-r(S)$.
Putting all of the above steps together, we get \[\Cov(S, X_i|S<\lambda)=(v_i-u_i)\Cov(S, B|S<\lambda)= (v_i-u_i)\Pb[X_i=u_i|S<\lambda](r(S_i)-r(S)).\qedhere\]
\end{proof}
\begin{lemma}\label{lemma:two_values}
Let $X_i$ be a random variable whose support includes at least three different values. Then, $X_i$ can be replaced by another random variable whose support includes at most two of the values from the support of $X_i$ and has the same mean. Moreover, $\Pb[\sum_{j=1}^n X_j<n+\delta]$ does not increase through this replacement.
\end{lemma}
\begin{proof}
Let $T=\sum_{j\ne i}X_j$. By standard approximation arguments, we may assume $X_i$ has finite support, say $v_1, \dots, v_m$. Let us denote $p_k=\Pb[X_i=v_k]$ and $q_k=\Pb[T+v_k<n+\delta]$. By the law of total probability, we have 
\[\Pb\Big[\sum_{j=1}^n X_j<n+\delta\Big]=\sum_{k=1}^m \Pb[X_i=v_k]\Pb\big[T+v_k<n+\delta\big]=\sum_{k=1}^m p_kq_k.\]
Hence, our goal is to minimize the linear program $\sum_{k=1}^m p_kq_k$ over the variables $p_1, \dots, p_m$, while preserving $\bE X_i=1$. This means that our program is subject to constraints $\sum_{k=1}^m p_k=1$ and $\sum_{k=1}^m p_kv_k=1$. The program is feasible (since we already know $X_i$ satisfies the constraints), and since it has two constraints, its optimal solution can have at most two nonzero values among $p_1, \dots, p_m$, meaning that the support of $X_i$ can be reduced to only two values without increasing $\Pb[\sum_{j=1}^n X_j<n+\delta]$.\hspace{5cm}$\square$
\end{proof}
\renewcommand{\qedsymbol}{}
\end{proof}
\vspace{-1.2cm}

\paragraph{Further remarks.} First, our theorem addresses the regime $\delta\geq 1$, since that is the relevant regime for Theorem~\ref{thm:fractional_matching}. However, in general, for any $\delta>0$, we can show that under the assumptions of Theorem~\ref{thm:feige}
\begin{equation}\label{eq:tight_feige}
    \Pb[X_1+\dots+X_n<n+\delta]\geq \min\Big\{\frac{\delta}{1+\delta}, \Big(1-\frac{1}{n+\delta}\Big)^n\Big\}.
\end{equation}
Indeed, the only place where the proof used the assumption $\delta\geq 1$ is to argue that $(1-1/(b+\delta))^b$ is minimized for $b=n$. In general, on can show that the function $(1-\frac{1}{b+\delta})^b$ always attains its minimum either at $b=1$ or at $b=n$, giving the bound (\ref{eq:tight_feige}).

The bound (\ref{eq:tight_feige}) is best possible, as the following two examples show. In the first example, take each $X_i$ to be either $0$ with probability $1-\frac{1}{n+\delta}$ or $n+\delta$ with probability $\frac{1}{n+\delta}$. Then, we have $\sum_{i=1}^n X_i<n+\delta$ only when all of the $X_i$ equal $0$, which happens with probability $\big(1-\frac{1}{n+\delta}\big)^n$. The other extremal example consists of all but one of $X_i$ being deterministically $1$, while the last one equals $0$ with probability $\frac{\delta}{1+\delta}$ and $1+\delta$ with probability $\frac{1}{1+\delta}$. The sum $\sum_{i=1}^n X_i$ is then less than $n+\delta$ precisely when the last random variable is $0$, i.e. with probability $\frac{\delta}{1+\delta}$.



\end{document}